\documentclass[12pt,a4paper]{amsart}
\usepackage[colorlinks=true]{hyperref}
\hypersetup{urlcolor=blue, citecolor=blue, draft=false}
\usepackage{hyperref}
\usepackage{color}
\allowdisplaybreaks

\usepackage{amsmath,amsthm,amsfonts,amssymb,graphicx,layout,indentfirst,float,fixltx2e}
\usepackage{subfigure}
\usepackage{epstopdf} 

\newtheorem{theorem}{Theorem}[section]

\newtheorem{corollary}{Corollary}[section]
\theoremstyle{theorem}

\newtheorem{proposition}{Proposition}[section]
\newtheorem{lemma}{Lemma}[section]
\theoremstyle{definition} \theoremstyle{definition}
\newtheorem{defn}{Definition}[section]

\newcommand{\ds}{\displaystyle}

\numberwithin{equation}{section}
\makeatletter
\@namedef{subjclassname@2010}{2010 Mathematics Subject Classification} \makeatother

\begin{document}
\title[A subclass of harmonic univalent mappings with a restricted analytic part]{A subclass of harmonic univalent mappings with a restricted analytic part}
\author[B. K. Chinhara]{B. K. Chinhara}
\author[P. Gochhayat]{P. Gochhayat}
\address{Department of Mathematics \\
	Sambalpur University\\
	Jyoti Vihar, 768019\\
	Burla, Sambalpur, Odisha\\ India}
\email{bikashchinhara@ymail.com}
\email{pgochhayat@gmail.com}
\author[S. Maharana ]{S. Maharana }
\address{Discipline of Mathematics, School of Basic Sciences \\
	Indian Institute of Technology Indore\\
	Khandwa Rd, Simrol, Madhya Pradesh 453552\\
	\\ India}
\email{smaharana@gmail.com}
\date{}

\begin{abstract}
In this article, a subclass of univalent harmonic mapping is introduced by restricting its analytic part to lie  in the class $\mathcal{S}^{\delta}[\alpha]$, $0\leq \alpha < 1$, $-\infty < \delta < \infty$ which has been introduced and studied by Kumar \cite{Kumar87} (see also \cite{Mishra95}, \cite{MishraChoudhury95},  \cite{MishraDas96}, \cite{MishraGochhayat06}). Coefficient estimations, growth and distortion properties, area theorem and covering estimates of functions in the newly defined class have been established. Furthermore, we also found bounds for the Bloch's constant for all functions in that family.
\end{abstract}


\maketitle{}
\section{Introduction}\label{sec1}
Let $\mathcal{A}$ denote the class of functions $f$ of the form $f(z)=z+a_2z^2+a_3z^3+\cdots$, which are holomorphic in the open unit disc $\mathbb{D}:=\{z: z\in\mathbb{C}~\text{and}~|z|<1\}$. The subclass of $\mathcal{A}$ consisting of all holomorphic and univalent functions in $\mathbb{D}$ will be denoted by $\mathcal{S}$. A well known sufficient condition (cf. \cite{Titchmarsh}) for function to be in the class $\mathcal{S}$ is that $\sum_{n=2}^\infty n|a_n|\leq 1$. An analogous sufficient condition (cf. \cite{Kumar87}) for the function $f$ to be in the class $\mathcal{S}^\delta[\alpha]$, $0\leq \alpha < 1$, $-\infty < \delta < \infty$ is that $\sum_{n=2}^\infty n^\delta\left({n-\alpha\over 1-\alpha}\right)|a_n|\leq 1$. Note that for each fixed $n$ the function $n^\delta$ is increasing with respect to $\delta$. This shows that if $\delta$ increases then the corresponding class decreases. Consequently, the functions in $\mathcal{S}^\delta[\alpha]$ are univalent starlike of  order $\alpha$ if $\delta \geq 0$ and if $\delta \geq 1$ then the  functions in the family $\mathcal{S}^\delta[\alpha]$ are univalent convex of order $\alpha$. The classical subfamily $\mathcal{C}(\alpha)~(\mathcal{S}^*(\alpha))$, of univalent convex of order $\alpha$ (univalent starlike of  order $\alpha,~0\leq \alpha < 1)$, respectively were introduced and studied in \cite{Robertson} are found in  \cite{duren1, goodman,GrahamKohr03}. Mishra and Choudhury \cite{MishraChoudhury95} showed that the class $\mathcal{S}^\delta[\alpha]$ also contains non-univalent function for negative $\delta$. In \cite{MishraGochhayat06},  Mishra and Gochhayat found the coefficient estimate problems for inverse of functions in the class $\mathcal{S}^\delta[\alpha]$. For more basic properties of the class we refer \cite{MishraDas96, MishraGochhayat06} and the references therein. In the present sequel, it is our interest to consider planner harmonic mappings whose analytic part being a member of the family $\mathcal{S}^\delta[\alpha]$.

Recalling, a complex-valued harmonic function in a simply connected domain $D$ subset of the complex plane $\mathbb{C}$  has a representation $f = h + \overline{g}$, where $h$ and $g$ are analytic functions in $D$, that is unique up to an additive constant. Notice that when $D = \mathbb{D}$, the open unit disk, it is convenient to choose the additive constant so that $g(0) = 0$. The representation $f = h + \overline{g}$ is therefore unique and is called the canonical representation of $f$. Lewy in \cite{Lewy36} proved that $f$ is locally univalent if and only if the Jacobian satisfies $J_f = |h'|^2 - |g'|^2 \ne 0$. Thus, harmonic mappings are either sense-preserving or sense-reversing depending on the conditions $J_f > 0$ and $J_f < 0$ respectively throughout the domain $D$, where $f$ is locally univalent. Since $J_f > 0$ if and only if $J_{\bar{f}} < 0$ we will consider sense-preserving mappings in $\mathbb{D}$ throughout all of this work. In this case the analytic part $h$ is locally univalent in $\mathbb{D}$ since $h^\prime\ne 0$, and
the second complex dilatation $w$ of $f$, given by $w = g^\prime/h^\prime$, is an analytic function in $\mathbb{D}$ with $|w| < 1$, see \cite{ClunSheil}.

Let $\mathcal{H}$ denote the set of locally univalent and sense preserving complex harmonic mappings in $\mathbb{D}$. Therefore, all function $f$ in the class $\mathcal{H}$ has unique power series representation of the form:
\begin{equation}\label{Veq1}
f=h+\overline{g},\quad \text{where}~\quad h(z)=\sum_{n=0}^{\infty}a_nz^n \quad {\rm and} \quad g(z)=\sum_{n=1}^{\infty}b_nz^n \qquad (z\in {\mathbb{D}}),
\end{equation}
and $a_n, b_n \in \mathbb{C}$. Following Clunie and Sheil-Small's notation \cite{ClunSheil}, let  $\mathcal{S}_{\mathcal{H}}\subset \mathcal{H}$ denote  the class of all sense preserving univalent harmonic functions  $f=h+\overline{g}$ in $\mathbb{D}$ with the normalization $h(0)=g(0)=h'(0)-1=0$. The class $\mathcal{S}_{\mathcal{H}}$ is a normal family \cite{durenbook}. Further, the subclass of $\mathcal{S}_{\mathcal{H}}$ for which $g^\prime (0)=0$ is denoted by $\mathcal S_{\mathcal H}^{0}$   is a compact normal family of harmonic functions \cite{ClunSheil}.
Moreover, when $f\in \mathcal S_{\mathcal H}$, the coefficients $a_0=0$ and $a_1=1$, and for $f \in \mathcal S_{\mathcal H}^{0}$, the coefficients $a_0=0,~ a_1=0$, and  $b_1=0.$

It is pertinent that for a fixed analytic function $h$ an interesting problem arises to describe all functions $g$ such that $f\in\mathcal H$. Not much known on the geometric properties of such planner harmonic functions. Klimek and Michalski \cite{KM}, first studied the properties of a subset of $\mathcal S_{\mathcal H}$ which is defined for all univalent anti-analytic perturbation of the identity and also considered the subclass of $\mathcal S_{\mathcal H}$ which is defined by restricting $h$ as a member of $\mathcal C$, univalent convex functions \cite{KM1}. Very recently, Hotta and Michalski \cite{HottaMichalski}  considered $h$ as a member of $\mathcal S^*$, univalent starlike functions  and discuss some geometric properties of certain subfamily of $\mathcal S_{\mathcal H}$. Few more subclasses of planner harmonic mappings were considered by restricting $h$ as member of univalent starlike function of order $\alpha$ \cite{ZhuHuang15}, univalent convex function of order $\alpha$  \cite{ZhuHuang15} and  to be in the class of bounded boundary rotation \cite{KanasKlimek14}.  In continuation to our earlier works, in the present sequel we considered the following new subclass of $\mathcal S_{\mathcal H}$ which is defined by considering $h$ as a member of $\mathcal S^\delta[\alpha]$. Thus, we define following:
\begin{defn} A function $f\in\mathcal{H}$ given in \eqref{Veq1} is said to be in the class $\mathcal{S}_{H}^{\delta}[\alpha, \beta]$, with $\alpha, \beta \in [0,1)$ and $-\infty < \delta < \infty$ if the analytic part of $f$ is a member of $\mathcal{S}^\delta[\alpha]$ and $|b_1|=\beta$. Equivalently,
	$$\mathcal{S}_H^\delta[\alpha, \beta]:=\{f=h+\overline{g}\in \mathcal{H}: h(z) \in \mathcal{S}^\delta[\alpha]~\text{and}~ |b_1|=\beta;~\alpha, \beta \in [0,1);~-\infty < \delta < \infty \}.$$
\end{defn}
Some facts about $\mathcal{S}_H^\delta[\alpha, \beta]$:
\begin{itemize}
	\item[(i)] for $\delta \geq 1,  $ the analytic part of $f\in \mathcal{S}_H^\delta[\alpha, \beta]$ is a convex function of order $\alpha$ (see. \cite{MishraGochhayat06}), therefore in view of \cite[Theorem 5.7]{ClunSheil} and the identity  $|w(z)|=|g'(z)/h'(z)|<1$, $f$ is sense preserving and normalized univalent function. Thus $\mathcal{S}_H^\delta[\alpha, \beta]\subset \mathcal{S}_{\mathcal H}$. Further for $\beta=0$, along with $\delta \geq 1$, we have $\mathcal{S}_H^\delta[\alpha, \beta]\subset \mathcal{S}_{\mathcal H}^0$.
	\item[(ii)] for $\delta \geq 1$, the class $\mathcal{S}_H ^{\delta}[\alpha, \beta]$ becomes a subclass of $\mathcal{S}_H ^{\beta}[C_{\alpha}]$ studied in \cite{ZhuHuang15}.
	\item[(iii)] for $\delta \geq 0$,  the analytic part of $f\in \mathcal{S}_H^\delta[\alpha, \beta]$ becomes a starlike function of order $\alpha$. Therefore $\mathcal{S}_H^\delta[\alpha, \beta]$ is a subclass of $\mathcal{S}_H^\beta[S_{\alpha}]$, which is introduced in \cite{ZhuHuang15}.
\end{itemize}
Finally, it is an easy exercise to check that 
\begin{proposition}\label{p1}
	For $\beta=0, \delta \geq 0, 0\leq \alpha <1$, the class $\mathcal{S}_H^{\delta}[\alpha, \beta]$ is a convex set.
\end{proposition}
%
%
%
\medskip
In the present investigation we also need following definitions and notations.

A harmonic function $f$ is called a \emph{Bloch mapping}, if and only if
\begin{equation}
\mathcal{B}_f = \sup_{z,\xi\in\mathbb{D},~ z\neq \xi}
\frac{|f(z)-f(\xi)|}{\mathcal{Q}(z,\xi)} < \infty,
\end{equation}
where
\[\mathcal{Q}(z,\xi)= \frac{1}{2} \log \left(\frac{1+\left|\frac{z-\xi}{1-\overline{z}\xi}\right|}{1-\left|\frac{z-\xi}{1-\overline{z}\xi}\right|}\right)={\rm arctanh}\left|\frac{z-\xi}{1-\overline{z}\xi}\right|\]
denotes the hyperbolic distance between $z$ and $\xi$ in
$\mathbb{D}$, and $\mathcal{B}_f$ is called the \emph{Bloch's
	constant} of $f$. In \cite{colonna}, Colonna proved that the Bloch's
constant $\mathcal{B}_f$ of a harmonic mapping $f=h+\overline{g}$
can be expressed as follows:
\begin{eqnarray}\label{Bloch}
\mathcal{B}_f 
&=& \sup_{z\in\mathbb{D}} \left(1-|z|^2\right) (|h'(z)|+|g'(z)|) \notag\\
&=& \sup_{z\in\mathbb{D}} \left(1-|z|^2\right) |h'(z)| (1+|w(z)|),
\end{eqnarray}
which agrees with the well-known notion of the Bloch's constant for analytic functions. Moreover, the set of all harmonic Bloch mappings forms a complex Banach space with the norm $\|\cdot\|$ given by
\[\|f\|=|f(0)|+\sup_{z\in\mathbb{D}} (1-|z|^2) \Lambda_f(z)\],
where
\[ \Lambda_f = \max_{0\leq \theta \leq 2\pi} \left|f_z(z)-e^{-2i\theta}f_{\overline{z}}(z)\right| = |f_z(z)|+|f_{\overline{z}}(z)|= |h'(z)|+|g'(z)|=|h'(z)| (1+|w(z)|). \]
This definition agrees with the notion of the Bloch's constant for analytic functions. Recently, some  authors (see \cite{chen,chen1,kanas1,liu}) have studied Bloch's constant for harmonic mappings. Thus it is our interest to find the bound on Bloch constant for functions in the family  $\mathcal{S}_H^\delta[\alpha, \beta]$.

The following Lemmas are being used to proof our main results.
\begin{lemma}\label{Vlemma1}\cite{GrahamKohr03}
	If $\Phi(z)=c_0+c_1z+c_2z^2+\cdots$ is analytic function and $|\Phi(z)|\leq 1$ on the open unit disk $\mathbb{D}$, then $|c_n|\leq 1-|c_0|^2$, $n=1, 2, 3, \cdots.$
\end{lemma}

\begin{lemma}\label{lemma2}\cite{MishraGochhayat06}
	If the function $f\in \mathcal{S}^\delta[\alpha]$, then
	\[ |a_n|\leq \frac{(1-\alpha)}{n^{\delta}(n-\alpha)}, \quad n=2,3,\cdots,\]
	and equality holds for each $n$ only for functions of the form
	\[f_n(z)= z+ \frac{(1-\alpha)}{n^\delta (n-\alpha)} e^{i\theta} z^n, \quad \theta\in\mathbb{R}.\]
\end{lemma}

\begin{lemma}\label{new-lemma1}\cite{vincent}{\rm( For bisection version see \cite{akritas, alesina})}
	Let $p(x)$ be a polynomial of degree $n$. There exists a positive quantity $\delta$ so that for every pair of positive rational numbers $a, b$ with $|b -a| < \delta$ every transformed polynomial of the form
	\[ V(x) = (1+x)^n \,p \left(\frac{a+bx}{1+x}\right) \]
	has exactly $0$ or $1$ variations in the sequence of its coefficients. The second case is possible if and only if $p(x)$ has a simple root within $(a, b).$ Moreover, the number of the sign variation is the maximal number of roots in $(a, b).$
\end{lemma}
\section{Main results}

\subsection{Coefficient Bound.}
In this section we have studied the bound of $|b_n|$, for\linebreak $f=h+\overline{g}\in \mathcal{S}_H^\delta[\alpha, \beta]$, with $\delta \geq 0$, where $h$ and $g$ have the series representation of the form \eqref{Veq1}.
\begin{theorem}\label{Vthem1}
	Let $f(z)=h(z)+\overline{g(z)}\in \mathcal{S}_{H}^{\delta}[\alpha, \beta]$, $\delta \geq 0$, where $h(z)$ and $g(z)$ are given by \eqref{Veq1}. Then
	\begin{equation}\label{Veq2bar}
	|b_n|\leq \begin{cases}
	\dfrac{(1-\alpha)\beta}{2^\delta(2-\alpha)}+\dfrac{(1-\beta^2)}{2}, & n=2,
	\\
	\dfrac{(1-\alpha)(1-\beta^2)}{n} \ds\sum_{k=1}^{n-1}\frac{k^{1-\delta}}{k-\alpha}+\dfrac{(1-\alpha)\beta}{n^{\delta}(n-\alpha)}, &  n=3,4,\cdots.
	\end{cases}
	\end{equation}
\end{theorem}
\begin{proof}
	Let the function $f(z)=h(z)+\overline{g(z)}$ be in the class $\mathcal{S}_{H}^{\delta}[\alpha, \beta]$, where $h$ and $g$ are represented by \eqref{Veq1}. Let $g'(z)=w(z)h'(z)$, where $w(z)$ is the dilatation of $f$, 
	which is analytic in $\mathbb{D}$ and has power series representation of the form
	\begin{equation}\label{Veq3}
	w(z)=\sum_{n=0}^\infty c_nz^n \qquad (z\in \mathbb{D}),
	\end{equation}
	where $c_n\in \mathbb{C}$. Clearly,  $c_0=|w(0)|=|g'(0)|=|b_1|=\beta<1$. Further, since $f\in\mathcal{S}_{H}^{\delta}[\alpha, \beta]$ is sense preserving, we have $|w(z)|<1$, for all $z\in \mathbb{D}$. Therefore, from Lemma \ref{Vlemma1}, we have
	\begin{equation*}
	|c_n|\leq 1-|c_0|^2, \quad n=1,2,\cdots.
	\end{equation*}
	Simplifying  $g'(z)=w(z)h'(z)$, by using the relations \eqref{Veq1} and \eqref{Veq3}, we have
	\begin{equation}\label{Veq4}
	\sum_{n=1}^{\infty} nb_n z^{n-1}=\sum_{n=0}^\infty\left(\sum_{k=0}^{n-1}(k+1)a_{k+1}c_{n-k-1} \right)z^{n-1}.
	\end{equation}
	Comparing the coefficients in \eqref{Veq4}, we obtain
	\begin{equation}\label{Veq5}
	nb_n=\sum_{k=0}^{n-1}(k+1)a_{k+1}c_{n-1-k}, \quad n=2,3,\cdots.
	\end{equation}
	Since $h(z)\in\mathcal{S}_H^\delta[\alpha]$, it is clear from Lemma \ref{lemma2} that
	\begin{equation}\label{Veq5a}
	|a_n|\leq \frac{1-\alpha}{n^{\delta}(n-\alpha)},\quad n=2,3,4,\cdots.
	\end{equation}
	Applications of \eqref{Veq5a} in \eqref{Veq5} together with Lemma \ref{Vlemma1} gives
	\begin{eqnarray*}
		n|b_n| &\leq & \sum_{k=0}^{n-2}(k+1)|a_{k+1}| |c_{n-1-k}| + n|a_n| |c_0| \\
		&\leq & \sum_{k=0}^{n-2}\frac{(p+1)(1-\alpha)(1-\beta^2)}{(k+1)^\delta (k+1-\alpha)}+ \frac{n(1-\alpha)\beta}{n^\delta (n-\alpha)},
	\end{eqnarray*}
	which implies that
	\[|b_n|\leq \frac{(1-\alpha)(1-\beta^2)}{n}\sum_{k=1}^{n-1}\frac{k^{1-\delta}}{k-\alpha}+ \frac{(1-\alpha)\beta}{n^\delta (n-\alpha)}.\]
	
	In particular for $n=2$, we have
	\[ 2|b_2| \leq  2|a_2| |c_0|+|a_1| |c_1| \leq \frac{2(1-\alpha)\beta}{2^\delta (2-\alpha)} +1-\beta^2 ,\]
	which together with Lemma \ref{Vlemma1} and Lemma \ref{lemma2}, provides
	\[|b_2|\leq \frac{(1-\alpha)\beta}{(2-\alpha) 2^\delta} + \frac{(1-\beta^2)}{2}. \]
	This completes the proof of Theorem \ref{Vthem1}.
\end{proof}
\begin{corollary}
	Let $f(z)=h(z)+\overline{g(z)}\in \mathcal{S}_{H}^{\delta}[\alpha, \beta]$, $\delta \geq 0$, where $h(z)$ and $g(z)$ are given by \eqref{Veq1}. Then, for $\beta=0$ and $\delta=1$, the coefficient of the co-analytic part of $f$ are:  $|b_1|=0,~ |b_2|\leq \frac{1}{2}$ and
	\[|b_n|\leq \left(\frac{1-\alpha}{n}\right)[\Psi(n-\alpha)-\Psi(1-\alpha)], n=3,4,\cdots,\] where $\Psi(x)$ represents the Psi (or Digamma) function of a real non-negative $x$, and is defined by the logarithmic derivative of  the usual Gamma function (cf.\cite{Gonzalez}).
	
\end{corollary}
\subsection{Growth and Distortion Results.}
In this section, we found the growth and distortion estimates of the analytic and co-analytic part of function $f$ in the class $\mathcal{S}_{H}^{\delta}[\alpha, \beta].$

\begin{theorem}\label{Vthem2}
	Let $f(z)=h(z)+\overline{g(z)}\in\mathcal{S}_H^\delta[\alpha, \beta]$, $\delta \geq 0$, where $h(z)$ and $g(z)$ are given by \eqref{Veq1}. Then for $z=re^{i\theta},~\theta\in\mathbb{R}$, we have
	\begin{equation}\label{Veq6bar}
	1-\frac{(1-\alpha)r}{(2-\alpha)2^{\delta-1}} \leq |h'(z)|\leq 1+ \frac{(1-\alpha)r}{(2-\alpha)2^{\delta-1}}
	\end{equation}
	and
	\begin{equation}\label{Veq7bar}
	\left(\frac{|\beta-r|}{1-\beta r} \right)\left(1-\frac{(1-\alpha)r}{(2-\alpha)2^{\delta-1}} \right)\leq |g'(z)| \leq \left(
	\frac{\beta+r}{1+\beta r}\right)\left(1+\frac{(1-\alpha)r}{(2-\alpha)2^{\delta-1}} \right).
	\end{equation}
\end{theorem}

\begin{proof}
	Let $g'(0)=\beta e^{i\mu},~\mu$ real. From a given dilatation $w(z)$, $|w(z)|<1$, we consider
	\begin{equation*}
	F_0(z):=\frac{e^{-i\mu}w(z)-\beta}{1-\beta e^{-i\mu}w(z)},\quad z=re^{i\theta}\in \mathbb{D}.
	\end{equation*}
	As $F_0(z)$ satisfies the assumptions of the Schwartz lemma, therefore $|F_0(z)|\leq |z|$. That is,
	\begin{equation*}
	|e^{-i\mu}w(z)-\beta|\leq |z||1-\beta e^{-i\mu}w(z)|,\quad z=re^{i\theta}\in \mathbb{D},
	\end{equation*}
	which is equivalent to
	\begin{equation}\label{Veq6}
	\left| e^{-i\mu}w(z)-\frac{\beta(1-r^2)}{1-\beta^2r^2}\right|\leq \frac{r(1-\beta^2)}{(1-\beta^2r^2)},\quad z=re^{i\theta}\in \mathbb{D}.
	\end{equation}
	Equality in the above inequality holds for the function
	\begin{equation}
	w(z)=e^{i\mu}\frac{e^{i\phi} z+\beta}{1+\beta e^{i\phi}z}, \qquad z\in \mathbb{D},\;\; \phi \in \mathbb{R}.
	\end{equation}
	Applying triangle inequality over \eqref{Veq6}, we obtain
	\begin{equation}\label{Veq7}
	\frac{|\beta-r|}{1-\beta r}\leq |w(z)|\leq \frac{\beta+r}{1+\beta r},\quad |z|=r<1.
	\end{equation}
	The function $f=h+\overline{g}\in \mathcal{S}_H ^{\delta}[\alpha, \beta]$, indicates that $h(z)\in S^{\delta}[\alpha]$. Hence,
	\begin{equation}\label{Veq8bar}
	\sum_{n=2}^{\infty}n^{\delta}\left(\frac{n-\alpha}{1-\alpha}\right)|a_n|\leq 1.
	\end{equation}
	Clearly,
	\begin{equation*}
	(2-\alpha)\sum_{n=2}^{\infty}n^{\delta}|a_n| \leq \sum_{n=2}^{\infty} n^{\delta}(n-\alpha)|a_n|\leq (1-\alpha).
	\end{equation*}
	Therefore,
	\begin{equation}\label{Veq8}
	\sum_{n=2}^{\infty}n^{\delta}|a_n|\leq \frac{1-\alpha}{2-\alpha}.
	\end{equation}
	
	For $\delta \geq 0$, $n^{\delta}$ is increasing in $n$. Thus using (\ref{Veq8bar}) and \eqref{Veq8}, we get
	\begin{align*}
	2^{\delta}\sum_{n=2}^{\infty}n|a_n|\leq \sum_{n=2}^{\infty}n^{\delta}n|a_n|
	&=\sum_{n=2}^{\infty}n^{\delta}(n-\alpha)|a_n|+\sum_{n=2}^{\infty}\alpha n^{\delta}|a_n|\\
	&\leq (1-\alpha)+\alpha\left(\frac{1-\alpha}{2-\alpha}\right).
	\end{align*}
	Therefore,
	\begin{equation}\label{Veq9bar}
	\sum_{n=2}^{\infty}n|a_n|\leq \frac{1-\alpha}{2^{\delta-1}(2-\alpha)}.
	\end{equation}
	
	Consider the function
	\begin{equation*}
	G(z):=zh'(z)=z+\sum_{n=2}^{\infty}n a_n z^n \qquad (z\in\mathbb D).
	\end{equation*}
	Therefore, using (\ref{Veq9bar})
	\begin{equation*}
	|G(z)|=|zh'(z)|\leq |z|+\sum_{n=2}^{\infty}n|a_n||z|^n\leq r+r^2\frac{1-\alpha}{2^{\delta-1}(2-\alpha)}.
	\end{equation*}
	This gives right hand side estimates of (\ref{Veq6bar}).
	
	Similarly,
	\begin{equation*}
	|G(z)|=|zh'(z)|\geq |z|-\sum_{n=2}^{\infty}n|a_n||z|^n \geq r- r^2\frac{1-\alpha}{2^{\delta-1}(2-\alpha)},
	\end{equation*}
	which estimates the left hand side of the inequality (\ref{Veq6bar}).
	
	By using \eqref{Veq7} and (\ref{Veq6bar}), in the identity $g'(z)=w(z)h'(z)$, we have
	\begin{equation}\label{Veq9}
	|g'(z)|\leq \left(\frac{\beta+r}{1+\beta r} \right)|h'(z)| \leq \left(\frac{\beta+r}{1+\beta r} \right)\left(1+\frac{(1-\alpha)r}{(2-\alpha)2^{\delta-1}} \right),
	\end{equation}
	and
	\begin{equation}\label{Veq10}
	|g'(z)|\geq \frac{|\beta-r|}{1-\beta r}h'(z)\geq \left(\frac{|\beta-r|}{1-\beta r}\right)\left(1-\frac{(1-\alpha)r}{(2-\alpha)2^{\delta-1}} \right).
	\end{equation}
	This completes the proof of Theorem \ref{Vthem2}.
\end{proof}

Following theorem gives the upper and lower bounds of the co-analytic part of function $f=h+\overline{g}$ which belonging to the  functions class $\mathcal{S}_H^{\delta}[\alpha, \beta]$, where $h$ and $g$ be of the form \eqref{Veq1}.

\begin{theorem}\label{Vthem3}
	Let $f(z)=h(z)+\overline{g(z)}\in \mathcal{S}_H ^{\delta}[\alpha, \beta]$, $\delta \geq 0$, where $h(z)$ and $g(z)$ be of the form \eqref{Veq1}. Then for $z=re^{i\theta},~\theta$ real, we have
	\begin{equation*}
	\left|\frac{F}{2^\delta(2-\alpha) \beta^3}\right| \leq |g(z)|\leq \frac{E}{2^\delta(2-\alpha) \beta^3} ,
	\end{equation*}
	where
	\begin{align*}& E:=r\beta \left(2^\delta (2-\alpha)\beta - (1-\alpha)(2-(r+2\beta)\beta)\right) + (2-2\alpha-2^\delta(2-\alpha)\beta)(1-\beta^2)\log(1+r\beta)\\
	& and \\
	& F:=r\beta \left(-2^\delta (2-\alpha)\beta + (1-\alpha)(2+(r-2\beta)\beta)\right) + (2-2\alpha-2^\delta(2-\alpha)\beta)(1-\beta^2)\log(1-r\beta).\end{align*}
\end{theorem}

\begin{proof}
	Choose a path $\gamma:=[0,z].$ From \eqref{Veq9} and \eqref{Veq10}, for $|z|=r<1$,  we have
	\begin{equation}\label{Veq10star}
	\left(\frac{|\beta-r|}{1-\beta r}\right)\left(1-\frac{(1-\alpha)r}{(2-\alpha)2^{\delta-1}} \right) \leq |g'(z)|\leq \left(\frac{\beta+r}{1+\beta r} \right) \left(1+\frac{(1-\alpha)r}{(2-\alpha)2^{\delta-1}} \right) .
	\end{equation}
	Right hand inequality of Theorem \ref{Vthem3} is obtained immediately upon integration along a radial line $\eta=te^{i\theta}$.
	
	In order to find the lower bound, let $\Gamma=g(\{z:|z|=r\})$ and let $\xi_1\in\Gamma$ be the nearest point to the origin. By a rotation we may assume that $\xi_1>0$. Let $\gamma$ be the line segment $0\leq\xi\leq\xi_1$ and suppose that $z_1=g^{-1}(\xi_1)$ and $L=g^{-1}(\gamma)$. with $\eta$ as the variable of integration on $L$, we have that $d\xi=g'(\eta)d\eta>0$ on $L$. Hence
	\begin{eqnarray*}
		\xi_1 &=& \int_0^{\xi_1} d\xi = \int_0^{z_1} g'(\eta)\;d\eta = \int_0^{z_1} |g'(\eta)| \,|d\eta| \geq \int_0^r |g'(te^{i\theta})|\; dt \\
		&\geq & \int_0^r \left(\frac{|\eta-\beta|}{1-\beta \eta}\right)\left(1-\frac{(1-\alpha)\eta}{(2-\alpha)2^{\delta-1}} \right) d\eta = \left|\frac{F}{2^\delta(2-\alpha) \beta^3}\right|.
	\end{eqnarray*}
	
	This completes the proof of Theorem \ref{Vthem3}.
\end{proof}

\subsection{Area Estimation.}
In this section, we deal with the area estimates of the function $f(z)\in \mathcal{S}_H^{\delta}[\alpha, \beta]$, where $h$ and $g$ be of the form \eqref{Veq1}.

\begin{theorem}\label{Vthem4}
	Let $f(z)=h(z)+\overline{g(z)}\in \mathcal{S}_H^{\delta}[\alpha, \beta]$, $\delta \geq 0$, where $h(z)$ and $g(z)$ be of the form \eqref{Veq1}. Then the estimation of area $A:=\int\int_{\mathbb{D}}J_f(z)dxdy,$ is given by
	\begin{multline}\label{Veq10a}
	2\pi \int_{0}^{1}r\left(\frac{((2-\alpha)^22^{2(\delta-1)}-r(1-\alpha))^2}{(2-\alpha)^22^{2(\delta-1)}} \right)\left(\frac{(1-\beta^2)(1-r^2)}{(1+\beta r)^2} \right)dr \leq A\\
	\leq 2\pi  \int_{0}^{1}r\left(\frac{((2-\alpha)^22^{2(\delta-1)}+r(1-\alpha))^2}{(2-\alpha)^22^{2(\delta-1)}} \right)\left(\frac{(1-\beta^2)(1-r^2)}{(1-\beta r)^2} \right)dr.
	\end{multline}
\end{theorem}

\begin{proof}
	Let $f(z)=h(z)+\overline{g(z)}\in \mathcal{S}_H ^{\delta}[\alpha, \beta] $, with the Jacobian $J_f(z)=|h^{'}(z)|^2(1-|w(z)|^2),$ where $w(z)$ is the dilatation of $f$, for $z\in \mathbb{D}.$
	
	By making use of estimates (\ref{Veq6bar}) and \eqref{Veq7}, we obtain
	\begin{align*}
	A:&=\int\int_{\mathbb{D}} J_f(z)\; dxdy=\int_{0}^{2\pi}d\theta\int_{0}^{1}J_f(re^{i\theta})r\;dr\\
	&=2\pi\int_{0}^{1} r J_f(re^{i\theta})\;dr=2\pi\int_{0}^{1}r|h'(re^{i\theta})|^2\left(1-|w(re^{i\theta})|^2\right)\;dr\\
	&\geq 2\pi \int_{0}^{1}r\left(1-\frac{(1-\alpha)r}{(2-\alpha)2^{\delta-1}}\right)^2\left(1-\left(\frac{\beta+r}{1+\beta r} \right)^2 \right)\;dr\\
	&=2\pi \int_{0}^{1}r\left(\frac{((2-\alpha)^2 2^{2(\delta-1)}-r(1-\alpha))^2}{(2-\alpha)^2 2^{2{(\delta-1)}}} \right)\left(\frac{(1-r^2)(1-\beta^2)}{(1+\beta r)^2} \right)\;dr,
	\end{align*}
	which is the left hand inequality of the estimates in \eqref{Veq10a}. Next to find the right hand estimates, we have
	\begin{align*}
	A:&=2\pi\int_{0}^{1}r|h'(re^{i\theta})|^2\left(1-|w(re^{i\theta})|^2\right)\;dr\\
	&\leq 2\pi \int_{0}^{1}r\left(1+\frac{(1-\alpha)r}{(2-\alpha)2^{\delta-1}}\right)^2\left(1-\left(\frac{\beta-r}{1-\beta r} \right)^2 \right)\;dr\\
	&=2\pi \int_{0}^{1}r\left(\frac{((2-\alpha)^2 2^{2(\delta-1)}+r(1-\alpha))^2}{(2-\alpha)^2 2^{2{(\delta-1)}}} \right)\left(\frac{(1-r^2)(1-\beta^2)}{(1-\beta r)^2} \right)\;dr.
	\end{align*}
	This completes the proof of Theorem \ref{Vthem4}.
\end{proof}

\subsection{Covering Results.}
In order to establish the covering result for the class $\mathcal{S}_H^{\delta}[\alpha, \beta]$, we first focused on the growth inequalities for the function $f(z)\in \mathcal{S}_H^{\delta}[\alpha, \beta]$ in the theorem stated below.

\begin{theorem}\label{Vthem5}
	Suppose that $f\in \mathcal{S}_H^{\delta}[\alpha, \beta]$, $\delta \geq 0$. The growth of $f$ is estimated as:
	\begin{equation}\label{Veq10bar}
	|f(z)|\geq \int_{0}^{r}\frac{((2-\alpha)2^{\delta-1}+(1-\alpha)\xi)(1-\beta)(1-\xi)}{(2-\alpha)2^{\delta-1}(1+\beta \xi)}d\xi
	\end{equation}
	and
	\begin{equation}\label{Veq10bar2}
	|f(z)|\leq r+\frac{r^2(1-\alpha)}{2^{\delta}(2-\alpha)}+\int_{0}^{r}\left(\frac{\beta+\xi}{1+\beta \xi} \right)\left(1+\frac{(1-\alpha)\xi}{(2-\alpha)2^{\delta-1}} \right)d\xi,
	\end{equation}
	for all  $z=re^{i\theta}\in \mathbb{D}.$
\end{theorem}

\begin{proof}
	Let $\mathbb{D}_r$ represents the disk of radius $r,\; r<1$, with center at origin. Denote $d:=\min\left\{|f(\mathbb{D}_r)|: z \in \mathbb{D}_r\right\}$. Clearly, $\mathbb{D}_d\subseteq f(\mathbb{D}_r)\subseteq f(\mathbb{D})$. It is due to minimum modulus principle that, there exists $t_r\in \partial{\mathbb{D}_r}:=\{z\in \mathbb{C}: |z|=r,\; r<1\}$, such that $d=|f(t_r)|.$
	
	Consider a Jordan arc as $l(t):=f^{-1}(L(t)), t\in[0,1]$, where $L(t):=tf(t_r).$
	
	\noindent Since $f\equiv f(z,\bar{z})=h(z)+\overline{g(z)}$, we have
	\begin{align}\label{Veq10bar3}
	d=|f(t_r)| &= \int_{L}|dw|=\int_{l}|df|=\int_l\left|f_{\eta}(\eta)\; d\eta+\overline{g_{\bar{\eta}}(\eta)} \;d\bar{\eta}\right|\\ \nonumber
	&\geq \int_{l}\left(|h'(\eta)|-|g'(\eta)|\right)|d\eta|.
	\end{align}
	From (\ref{Veq6bar}) and \eqref{Veq7}, the integrand of the above inequality became
	\begin{equation*}
	|h'(\eta)|-|g'(\eta)|=|h'(\eta)|(1-|w(\eta)|)\geq \left(1+ \frac{(1-\alpha)|\eta|}{(2-\alpha)2^{\delta-1}} \right)\left(1-\frac{\beta+|\eta|}{1+\beta |\eta|}  \right).
	\end{equation*}
	Therefore, \eqref{Veq10bar3} yields
	\begin{align}
	d &\geq \int_l \frac{((2-\alpha)2^{\delta-1}+(1-\alpha)|\eta|)(1-\beta)(1-|\eta|)}{(2-\alpha)2^{\delta-1}(1+\beta |\eta|)}|d\eta|\\ \nonumber
	&=\int_{0}^{1} \frac{((2-\alpha)2^{\delta-1}+(1-\alpha)|l(t)|)(1-\beta)(1-|l(t)|)}{(2-\alpha)2^{\delta-1}(1+\beta |l(t)|)}dt\\ \nonumber
	&\geq \int_{0}^{r}\frac{((2-\alpha)2^{\delta-1}+(1-\alpha)\xi)(1-\beta)(1-\xi)}{(2-\alpha)2^{\delta-1}(1+\beta \xi)}d\xi,~~~z=re^{i\theta}\in \mathbb{D},
	\end{align}
	which proofs the inequality (\ref{Veq10bar}).
	
	Next, we proceed to proof the inequality (\ref{Veq10bar2}). Note that
	\begin{equation}\label{Veq12}
	|f(z)|=|h(z)+\overline{g(z)}|\leq |h(z)|+|g(z)|.
	\end{equation}
	For $h\in\mathcal S^\delta[\alpha]$, it is an easy exercise to check that 
	\begin{equation}\label{Veq13}
	|h(z)|\leq r+\frac{r^2(1-\alpha)}{2^{\delta}(2-\alpha)},\qquad z=re^{i\theta}\in \mathbb{D}.
	\end{equation}
	Thus, the inequality (\ref{Veq10bar2}) follows from the applications of   (\ref{Veq10star}), (\ref{Veq12}) and (\ref{Veq13}).
	This completes the proof of Theorem \ref{Vthem5}.
\end{proof}
\begin{corollary}\label{corollary2}
	$\mathcal{S}_H^{\delta}[\alpha, \beta], \delta \geq 0; \alpha, \beta \in [0,1)$ is a normal family  of $\mathcal{H}$.
\end{corollary}
\begin{proof}
	Due to Montel's criterion for the normality of families of harmonic function, it is enough to prove that the class $\mathcal{S}_H^{\delta}[\alpha, \beta]$ is uniformly bounded in $\mathbb{D}$. That is for each $z_0\in \mathbb{D}$, there exists a constant $M>0$ and a neighborhood $N$ of $z_0$ such that $|f(z)|\leq M$ for each $f\in \mathcal{S}_H^{\delta}[\alpha, \beta] $ and $z_0\in N.$
	
	Let $f\in \mathcal{S}_H^{\delta}[\alpha, \beta], \delta \geq 0; \alpha, \beta \in [0,1).$ Then from the growth inequality (\ref{Veq10bar2}), it is clear that each function $f\in \mathcal{S}_H^{\delta}[\alpha, \beta]$ is uniformly bounded and where corresponding constant $M$ is obtained by letting $r\longrightarrow 1^-$ in  (\ref{Veq10bar2}), i.e.
	\begin{equation}
		M=1+\frac{1-\alpha}{2^{\delta}(2-\alpha)}+\int_{0}^{1}\left(\frac{\beta+\xi}{1+\beta \xi}\right)\left(1+\frac{(1-\alpha)\xi}{(2-\alpha)2^{\delta-1}} \right)d \xi.
	\end{equation}
	This completes the proof of Corollary \ref{corollary2}.
\end{proof}
The Covering theorem for $f\in \mathcal{S}_H ^{\delta}[\alpha,\beta]$ is established from the growth inequality (\ref{Veq10bar}), by the limiting approach of $r$ to $1^{-}$, and is stated as follows.

\begin{theorem}\label{Vthem6}
	The disk $\mathbb{D}$ is mapped by any function in $\mathcal{S}_H^{\delta}[\alpha,\beta]$, $\delta \geq 0;\alpha$, $\beta \in [0,1)$, onto a domain that contains the disk
	$$\left\{\omega \in \mathbb{C}:|\omega|<k:=\int_{0}^{1}\frac{((2-\alpha)2^{\delta-1}+(1-\alpha)\xi)(1-\beta)(1-\xi)}{(2-\alpha)2^{\delta-1}(1+\beta \xi)}d\xi,~~~ z=re^{i\theta}\in \mathbb{D}\right\}.$$
\end{theorem}

\medskip
\subsection{Estimate of Bloch's constant.}
In this section, we shall find bounds on the Bloch's constant for the functions $f=h+\overline{g}$ in the class $\mathcal{S}_H ^{\delta}[\alpha,\beta]$, $\delta \geq 0$; $\alpha, \beta \in [0,1)$.

\begin{theorem}\label{thm-new}
	Let $f=h+\overline{g}\in \mathcal{S}_H ^{\delta}[\alpha,\beta]$, where $h$ and $g$ are given by \eqref{Veq1}. Then the Bloch's constant $\mathcal{B}_f$ is bounded by
	\begin{equation}
	\mathcal{B}_f \leq \frac{(1+\beta)}{(2-\alpha)2^{\delta-1}} \; \frac{\left(1+r_0-r_0^2-r_0^3\right)\left((2-\alpha)2^{\delta-1}+(1-\alpha)r_0\right)}{(1+\beta r_0)},
	\end{equation}
	where $\,r_0\,$ is the only root of the equation
	\begin{eqnarray*}
		2^{\delta-1}(2-\alpha)(1-\beta)+(1-\alpha) - 2\left(2^{\delta-1}(2-\alpha)-(1-\alpha)\right)r \qquad\qquad\qquad\qquad\qquad\qquad \\
		- \left(2^{\delta-1}(2-\alpha)(3+\beta)+(1-\alpha)(3-\beta)\right)r^2 \qquad\qquad\qquad\qquad\qquad \\
		- \left(2^\delta(2-\alpha)\beta+(1-\alpha)(4+2\beta)\right)r^3-3(1-\alpha)\beta r^4 ,
	\end{eqnarray*}
	in the interval $(0,1)$.
\end{theorem}

\begin{proof}
	If  $f=h+\overline{g}\in \mathcal{S}_H ^{\delta}[\alpha,\beta]$, then $h\in \mathcal{S}_H ^{\delta}[\alpha]$. Using the distortion result from Theorem \ref{Vthem2} along with \eqref{Bloch} and \eqref{Veq7}, we obtain
	\begin{eqnarray*}
		\mathcal{B}_f &=& \sup_{z\in\mathbb{D}} \left(1-|z|^2\right) |h'(z)| \, (1+|w(z)|) \notag \\
		&\leq & \sup_{0\leq r<1} \left(1-r^2\right) \left(1+\frac{(1-\alpha)r}{(2-\alpha)2^{\delta-1}}\right)\left(1+\frac{r+\beta}{1+\beta r}\right)\\
		&=& \frac{(1+\beta)}{(2-\alpha)2^{\delta-1}} \sup_{0 \leq r<1} G(r),
	\end{eqnarray*}
	where
	\[ G(r):= \frac{(1+r-r^2-r^3)(2-\alpha)2^{\delta-1}+(1-\alpha)(r+r^2-r^3-r^4)}{(1+\beta r)}.\]
	The derivative of $G(r)$ is equal to zero, if and only if
	\begin{eqnarray*}
		2^{\delta-1}(2-\alpha)(1-\beta)+(1-\alpha) - 2\left(2^{\delta-1}(2-\alpha)-(1-\alpha)\right)r \qquad\qquad\qquad\qquad\qquad\qquad \\
		- \left(2^{\delta-1}(2-\alpha)(3+\beta)+(1-\alpha)(3-\beta)\right)r^2 \qquad\qquad\qquad\qquad\qquad \\
		- \left(2^\delta(2-\alpha)\beta+(1-\alpha)(4+2\beta)\right)r^3-3(1-\alpha)\beta r^4 =0,
	\end{eqnarray*}
	for $r\in(0,1)$. Denoting the last polynomial by $H(r)$, we note that
	\[H(0)= 2^{\delta-1}(2-\alpha)(1-\beta)+(1-\alpha) >0 \]
	and
	\[ H(1)=-4(1+\beta) \left(2^{\delta-1}(2-\alpha)+(1-\alpha)\right)<0,\]
	so that there exist $r_0\in(0,1)$, such that $H(r_0)=0$. Now it suffices to prove that $r_0$ is unique. To claim this it is enough to prove that  the derivative $H'(r) <0$ for $r\in(0,1)$ and $\beta\in(0,1)$. Let $r\in(0,1)$ be fixed, and we denote by $L(\beta)$ the derivative $H'(r)$, that is
	\[ L(\beta)= 2(1-\alpha)(1-3r-6r^2)-2^\delta(2-\alpha)(1+3r)+\left(-2^\delta(2-\alpha)(1+3r)r+2(1-\alpha)(1-3r-6r^2)r\right)\beta. \]
	One can see that
	\[L(0)=2(1-\alpha)(1-3r-6r^2)-2^\delta(2-\alpha)(1+3r) <0\]
	and
	\[L(1) = 2(1-\alpha)(1-2r-9r^2-6r^3)- 2^\delta (2-\alpha)(1+4r+3r^2) <0,\]
	for all $r\in (0,1)$. Denoting the coefficient of $L(\beta)$ by $a_0$ and $a_1$ we have
	\begin{align*}
	& a_0= 2(1-\alpha)(1-3r-6r^2)-2^\delta(2-\alpha)(1+3r) <0 \\
	& a_1= -2^\delta(2-\alpha)(1+3r)r+2(1-\alpha)(1-3r-6r^2)r <0.
	\end{align*}
	The sequence of the sign of $a_0$ and $a_1$ is $(-, -)$. Thus, There is no sign variations on $(0,1)$ for every $r\in(0,1)$ and  $\beta\in(0,1)$. 
	Thus, by the classical rule of Descartes-Harriot, there are no zeros of polynomial $L(\beta)$ in the interval $(0,1)$. It means that $L(\beta)<0$ for every $\beta\in(0,1)$ and $r\in(0,1)$, equivalently $H'(r)<0$ in $r\in(0,1)$. This is what we wanted to proof.
\end{proof}

\end{document}